\documentclass[a4paper, 10pt]{article}

\usepackage{amsmath,amssymb,amsthm}
\usepackage[latin1]{inputenc}
\usepackage{version,tabularx,multicol}
\usepackage{graphicx,float,psfrag}
\usepackage{stmaryrd}
 \usepackage{enumerate}
 \usepackage{color}
\usepackage[pdftex]{hyperref}
\usepackage{authblk}

\usepackage[numbers,sort&compress]{natbib}

\theoremstyle{plain}
\newtheorem{theorem}{Theorem}[section]

 \newtheorem{lemma}[theorem]{Lemma}

 \newtheorem{definition}[theorem]{Definition}

\newtheorem{remark}{Remark}[section]

  \makeatletter
  \@addtoreset{equation}{section}
  \makeatother

 \def\beqlb{\begin{eqnarray}}\def\eeqlb{\end{eqnarray}}
 \def\beqnn{\begin{eqnarray*}}\def\eeqnn{\end{eqnarray*}}

 \def\mbb{\mathbb}

 \def\qed{\hfill$\Box$\medskip}

 \def\E{\mathbb{E}}\def\P{\mathbb{P}}
 \def\lamb{\lambda}

\newcommand{\bcen}{\begin{center}}
\newcommand{\ecen}{\end{center}}
\newcommand{\bgeqn}{\begin{equation}}
\newcommand{\edeqn}{\end{equation}}

\def\E{{\mbb  E}}
\def\P{{\mbb  P}}

\renewcommand{\P}{\mathbb{P}}

\begin{document}

\title{Central limit theorems for a heavy-tailed subcritical Galton-Watson process with state-dependent immigration}

\author[a]{Xiaodan Luo\thanks{202521130105@bnu.edu.cn}}
\author[b]{Mei Zhang \thanks{Corresponding author: meizhang@bnu.edu.cn.}}
\affil{Laboratory of Mathematics and Complex Systems (Ministry of Education), School of Mathematical Sciences, Beijing Normal University, Beijing 100875, China}

\maketitle

\noindent\textit{Abstract:} In this paper, we study a subcritical Galton-Watson process with state-dependent immigration, where immigration occurs only if the preceding generation is empty. Under heavy-tailed assumptions on the offspring and immigration distributions, we establish central limit theorems for the total progeny up to time $n$, thereby extending existing results for the finite-variance model in the literature.
 
\bigskip


\bigskip


\section{Introduction}

We consider the branching process $\mathbf{X}=\{X_n:n\geq0\}$ defined by the following recursion:
\begin{align}\label{111}
X_n=\sum_{i=1}^{X_{n-1}}\xi_i^{(n)}+I_{n}\mathbf{1}_{\{X_{n-1}=0\}},\quad n\ge 1,\quad X_{0}=0,
\end{align}
where $\{\xi, \xi_i^{(n)}: n, i \geq 1\}$ is a sequence of independent and identically distributed (i.i.d.) non-negative integer-valued random variables. Here, $\xi_i^{(n)}$ denotes the number of offspring produced by the $i$-th individual in generation $n$, with common probability generating function (PGF) $f(s)=\sum_{j=0}^{\infty}p_{j}s^j$. The sequence $\{I, I_{n}:n\geq1\}$ is another i.i.d. sequence of non-negative integer-valued random variables, independent of $\{\xi_i^{(n)}\}$, where $I_n$ represents the potential number of immigrants entering generation $n$, with PGF $G(s)=\sum_{j=0}^{\infty}g_{j}s^j$. Let $\beta=\mathbb{E}I$ and $m=\mathbb{E}\xi$. The process $\mathbf{X}$ is termed supercritical (critical, or subcritical) if $m>1$ ($m=1$, or $m<1$, respectively). By construction, immigrants enter the population at generation $n$ if and only if $X_{n-1}=0$, and they subsequently evolve independently according to the offspring distribution $\{p_k:k\ge 0\}$.

This model was first introduced by Foster \cite{Foster1971} and Pakes \cite{Pakes71}. 
Foster \cite{Foster1971} investigated the asymptotic properties of $\mathbb{P}(X_{n}=0)$ and the limiting behavior of $\log X_{n}$ in the critical case. 
Pakes \cite{Pakes71} proved that the process $\mathbf{X}$ is recurrent when $m\leq 1$. 
Furthermore, when $m<1$, $\mathbf{X}$ admits a stationary distribution if and only if $\sum_{j=1}^{\infty} j h_j \log j < \infty$; the process is geometrically ergodic if $\sum_{j=1}^{\infty} j p_j \log j < \infty$ and $\sum_{j=1}^{\infty} j h_j < \infty$. 
Later, Kulkarni and Pakes \cite{BGT} studied the total population size of $\mathbf{X}$ up to time $n$. 
In particular, for the subcritical case ($m<1$), they established a central limit theorem under the assumptions that both the offspring and immigration distributions have finite variances, namely $\sum_{j=1}^{\infty} j^2 p_j < \infty$ and $\sum_{j=1}^{\infty} j^2 h_j < \infty$. 
Recently, Li et al. \cite{LLZ2026} obtained large deviations for the total population size under an exponential moment assumption. 
For the corresponding model in a random environment, Zhao and Zhang \cite{ZZ2025} studied the asymptotic behavior of the survival probability in the critical case ($m=1$). 
Moreover, for the subcritical Galton-Watson process with i.i.d. immigration, Guo and Hong \cite{gh2024} investigated large deviations for the total population size.

In this paper, we focus on the subcritical case, i.e., we always suppose $m<1$.  According to \cite[pp.1774, formula (2)]{Foster1971}, the PGF of $X_n$, denoted by $H_{n}(s):=\mathbb{E}(s^{X_{n}})$, is recursively given by
\begin{align}\label{gs}
H_{n}(s)=f_{n}(s)-\sum_{k=0}^{n-1}{[1-h(f_{k}(s))]H_{n-1-k}(0)},\quad n\geq1,
\end{align}
where $f_{k}(s)$ is the $k$-th iteration of $f(s)$, and $f_0(s)=s$ for $s\in[0,1)$.

We define the total population size of process $\mathbf{X}$ over the first $n$ generations as $S_n=\sum_{i=1}^n X_i$. Following Kulkarni and Pakes \cite{BGT}, we have the following decomposition:
\begin{align}\label{div1}
S_n=\sum_{i=1}^{K_n}U_i+\mathbf{1}_{\{B_n\neq \infty\}}\sum_{n-B_n}^n X_i.
\end{align}
Here, $K_n=\#\{1\leq \nu\leq n:X_\nu=0\}$ denotes the number of times the process $\mathbf{X}$ visits state $0$ from time $1$ to time $n$, and $B_n=n-\max\{\nu\leq n:X_{\nu}=0\}$ represents the time elapsed since the last visit to state $0$ before time $n$. We define
\begin{align}\label{tauU}
\tau_0=0,\;\tau_i=\min\{t>\tau_{i-1}:X_t=0\}, \quad i\ge 1, \quad
U_i=\sum\limits_{\tau_{i-1}<j\leq \tau_{i}}X_j
\end{align}
as the total population accumulated between the $(i-1)$-th and the $i$-th visits to state $0$. From the fundamental renewal theory, we have
\begin{align}\label{Kn0}
\dfrac{K_n}{n}\stackrel{\mathbb{P}}{\rightarrow} \frac{1}{\mathbb{E} \tau_1} :=\pi_0,\quad \dfrac{\E K_n}{n}\to   \pi_0,\quad n\to \infty.
\end{align}
Based on the decomposition \eqref{div1}, under the assumptions $\mathbb{E}I^2<\infty$ and $\mathbb{E}\xi^2<\infty$, Kulkarni and Pakes \cite[Theorems 2 and 3]{BGT} established two central limit theorems for $S_n$, where the limiting random variables follow normal distributions.

The objective of this paper is to relax the moment conditions on $I$ or $\xi$. We derive central limit theorems for $S_n$ when $I$ or $\xi$ follows a certain heavy-tailed distribution. To this end, we first introduce the following definitions:

\begin{definition}
If there exists a function $L$ slowly varying at infinity such that the random variable $Z$ satisfies
\begin{align}
\mathbb{P}(|Z|>x)\sim  x^{-\alpha}L(x), \quad x\rightarrow \infty,
\label{add}
\end{align}
then $Z$ is said to have a regularly varying tail with index $\alpha$.
\end{definition}

We define three conditions as follows:

{\bf (A)}\quad $I$ has a regularly varying tail such that $\mathbb{P}(I>x)\sim  x^{-\kappa}l_1(x)$, where $\kappa\in (1,2)$ and $l_1(x)$ is a function slowly varying at infinity; moreover, there exists $r>\kappa$ such that $\mathbb{E}\xi^r<\infty$.

 {\bf (B)}\quad $\xi$ has a regularly varying tail such that $\mathbb{P}(\xi>x)\sim  x^{-\alpha}l_2(x)$, where $\alpha\in (1,2)$ and $l_2(x)$ is a function slowly varying at infinity; moreover, there exists $r>\alpha$ such that $\mathbb{E}I^r<\infty$. 

 {\bf (C)}\quad $I$ has a regularly varying tail such that $\mathbb{P}(I>x)\sim  x^{-\gamma}l_3(x)$, where $\gamma\in (0,1)$ and $l_3(x)$ is a function slowly varying at infinity.

The following are the main theorems of this paper.  Theorems \ref{th1}--\ref{th4} concern the case that the dominant tail index is in $(1,2)$. 
 
 \begin{theorem}\label{th1}
Suppose   the process $\mathbf{X}=\{X_n:n\geq0\}$  is defined by \eqref{111} and satisfies {\bf (A)}. Then as $n\to \infty$, 
\begin{align*}
\frac{S_n-\beta K_n/(1-m)}{a_n}\stackrel{d}{\rightarrow} \pi_0^{-1/\kappa}\zeta_\kappa,
\end{align*}
  where  $\{a_n\}$ is an integer-valued sequence satisfying $na_n^{-\kappa}l_1(a_n)\to (1-m)^{\kappa}$,   and $\zeta_\kappa$ is a $\kappa$-stable random variable with characteristic function $$\psi(t)=\exp\left(- |t|^\kappa\Gamma(1-\kappa)\cos\frac{\pi\kappa}{2} (1+\mathrm{i}\mathrm{sgn}(t)\tan\frac{\pi\kappa}{2}) \right),\quad t\in\mathbb{R}.$$
\end{theorem}

 \begin{theorem}\label{th02}
 
Suppose the process $\mathbf{X}=\{X_n:n\geq0\}$ is defined by \eqref{111} and satisfies {\bf (B)}. Then as $n\to \infty$,  
\begin{align*}
\frac{S_n-\beta K_n/(1-m)}{a_n}\stackrel{d}{\rightarrow} \pi_0^{-1/\alpha}\zeta_\alpha,
\end{align*}
  where  $\{a_n\}$ is an integer-valued sequence satisfying $na_n^{-\alpha}l_2(a_n)\to\beta^{-1}(1-m)^{1+\alpha}$, and $\zeta_\alpha$ is an $\alpha$-stable random variable with characteristic function$$\psi(t)=\exp\left(- |t| ^\alpha\Gamma(1-\alpha)\cos\frac{\pi\alpha}{2} (1+\mathrm{i}\mathrm{sgn}(t)\tan\frac{\pi\alpha}{2}) \right),\quad t\in\mathbb{R}.$$
\end{theorem}

 \begin{remark}\label{remark01}

 To establish the aforementioned results, invoking the decomposition \eqref{div1} requires proving a central limit theorem for $\sum_{i=1}^{K_n}U_i$. This term represents a sum of independent random variables indexed by $K_n$, which is dependent on the summands $\{U_i\}$. While Kulkarni and Pakes \cite{BGT} proved their Theorem 2 using the Doeblin-Anscombe central limit theorem (Chow and Teicher \cite[Page 317]{C and T}) under a finite variance assumption, the present paper extends this framework to heavy-tailed random variables (see Theorem \ref{tuiguang}). Specifically, we adopt a truncation argument: the integral properties of slowly varying functions are utilized to handle the large-value tail, while Kolmogorov's maximal inequality is applied to bound the remaining part. Theorem \ref{tuiguang} provides the foundation for establishing Theorems \ref{th1} and \ref{th02}.

\end{remark}

 \begin{remark}\label{remark1}
By \cite[Theorem 1.5.13]{RV},  if $\ell$ is a slowly varying function at $\infty$, then there exists a slowly varying function $\ell^\#$ (called de Bruijn
conjugate of $\ell$),  unique up to asymptotic equivalence, such that
\begin{align*}
\ell(x)\ell^\#(x\ell(x))\to 1,\quad \ell^\#(x)\ell(x\ell^\#(x))\to 1,\quad x\to \infty,
\end{align*}
and $\ell^{\#\#}\sim \ell$.  In Theorem \ref{th02}, if we denote $\ell(x)=\beta^{-1/\alpha}(1-m)^{\frac{1+\alpha}{\alpha}}l_2^{-1/\alpha}(x)$, then by $a_n\ell(a_n)\sim n^{1/\alpha}$ we have $\ell(a_n)\sim \ell^\#(n^{1/\alpha})^{-1}$,   and therefore $a_n\sim n^{1/\alpha} \ell^\#(n^{1/\alpha})$. Clearly, $\mathcal{L}(x):=\ell^\#(x^ {1/\alpha})$ is a slowly varying function at $\infty$,  and $a_n\sim n^{1/\alpha}\mathcal{L}(n)$. That is to say, in asymptotic equivalence sense,  $a_n$ can be presented wth the help of  de Bruijn
conjugate.
\end{remark}

By the central limit theorem for renewal process, we can replace $K_n$ in Theorems \ref{th1} and \ref{th02} with its asymptotic mean. 

\begin{theorem}\label{th3}
Suppose   the process $\mathbf{X}=\{X_n:n\geq0\}$  is defined by \eqref{111} and satisfies {\bf (A)}. Then as $n\to \infty$,  
\begin{align*}
\frac{S_n-  n\beta\pi_0/(1-m)}{a_n}\stackrel{d}{\rightarrow} \pi_0^{-1/\kappa}\zeta_\kappa,
\end{align*}
where  $\{a_n\}$  and $\zeta_\kappa$ are the same as in Theorem \ref{th1}.  
\end{theorem}

\begin{theorem}\label{th4}
Suppose the process $\mathbf{X}=\{X_n:n\geq0\}$ is defined by \eqref{111} and satisfies {\bf (B)}. Then as $n\to \infty$,  
\begin{align*}
\frac{S_n-  n\beta\pi_0/(1-m)}{a_n}\stackrel{d}{\rightarrow} \pi_0^{-1/\alpha}\zeta_\alpha,
\end{align*}
where $a_n$ and $\zeta_\alpha$ are the same as in Theorem \ref{th02}.
\end{theorem}

The following result addresses the case where the dominant tail index lies in the interval $(0,1)$.

\begin{theorem}\label{1.6}

Suppose the process $\mathbf{X}=\{X_n:n\geq0\}$ is defined by \eqref{111} and satisfies {\bf (C)}. Then as $n\to \infty$,  
\begin{align*}
\frac{S_n}{a_n}\stackrel{d}{\rightarrow} \pi_0^{-1/\gamma}  T,
\end{align*}
where  $\{a_n\}$ is an integer-valued sequence satisfying $na_n^{-\gamma}l_3(a_n)\to(1-m)^\gamma\Gamma(1-\gamma)^{-1}$, and $T$ is a positive random variable with Laplace transform $\mathrm{e}^{-s^\gamma}$ ($s>0$).
\end{theorem}

\begin{remark}\label{remark03} Following Bingham et al. \cite[Subsection 8.3.5]{RV}, no centering is required when $0 < \gamma < 1$. In this regime, Theorem \ref{th3.6} plays the same crucial role  as Theorem \ref{tuiguang}.

\end{remark}

 \begin{remark}\label{remark04}
 
 Throughout this paper, we assume that the process $\mathbf{X}$ is initiated at $0$ (i.e., $X_0 = 0$), with the first immigrants entering at generation $1$. Alternatively, one could assume a general initial distribution, such as $X_0 \stackrel{d}{=} I$, implying that  immigrants arrive at generation $0$. This modification would alter the first hitting time of zero, rendering $\{K_n:n\ge 1\}$ a delayed renewal process. However, all analogous results remain valid in this setting.

\end{remark}

 In this paper, ``$X\stackrel{d}{=}Y$" denotes that $X$ and $Y$ have the same distribution; ``$\stackrel{d}{\longrightarrow}$" denotes convergence in distribution; ``$\stackrel{\mathbb{P}}{\longrightarrow}$" denotes convergence in probability; $f(n)=o(g(n))$ means $\lim_{n\to \infty} f(n)/g(n)=0$, and $f(n)\sim g(n)$ means $\lim_{n\to \infty} f(n)/g(n)=1$; $\operatorname{sgn}(x)=1$ if $x>0$, $\operatorname{sgn}(x)=-1$ if $x<0$, and $\operatorname{sgn}(0)=0$; $\lfloor x\rfloor$ denotes the greatest integer not exceeding $x$.


\section{Proofs of main results}

Let $\mathbf{Y}=\{Y_n:n\ge 0\}$ be a Galton-Watson branching process defined by $Y_0=1$ and
\[ Y_n=\sum_{j=1}^{Y_{n-1}}\xi_j^{(n)}, \quad n\ge 1, \]
where $\{\xi_i^{(n)}: n,i\geq1\}$ is the i.i.d. sequence sharing the common generating function $f$ as given in \eqref{111}. 
Define the total progeny of $\mathbf{Y}$ as $\tilde{Y}_\infty=\sum_{n=0}^\infty Y_n$, and let $\varphi(s):=\E[s^{\tilde{Y}_\infty-1}]$. 
According to Pakes \cite{Pakes71a}, we have the distributional identity 
\[ \tilde{Y}_\infty\stackrel{d}{=} 1 + \sum_{i=1}^{\xi} \tilde{Y}_\infty^{(i)}, \]
where $\{\tilde{Y}_\infty^{(i)}:i\ge 1\}$ is a sequence of independent copies of $\tilde{Y}_\infty$, independent of $\xi$. This leads to the functional equation
\begin{equation}\label{varphi}
\varphi(s) = f\big(s\,\varphi(s)\big), \qquad 0\le s\le1.
\end{equation}
Let $\mu = \E[\tilde{Y}_\infty-1]$. Differentiating \eqref{varphi} yields $\mu = m/(1-m)$, which implies
\begin{equation}\label{EYinfty}
\E \tilde{Y}_\infty=\frac{1}{1-m}.
\end{equation} 
We have

 \begin{lemma}\label{xiy0} For each $1<\alpha<2$, 
if $\mathbb{E}(\xi^{\alpha})<\infty$, then $\mathbb{ E}(\tilde{Y}_\infty^{\alpha})<\infty$. 
 
 \end{lemma}
 
 \begin{proof}    Guo and Hong \cite[Lemma 2.1]{gh2024} stated this result and provided a proof sketch. For the sake of completeness, we provide a detailed proof here. 

Define $y_n:=||Y_n||_{\alpha}=(\mathbb{E}[Y_n^\alpha])^{1/\alpha}$. Clearly, $y_0=1$.  By Lemma \ref{lem:minkowski} we have
\begin{align}\label{yn0}
y_n=\left(\mathbb{E}[Y_n^\alpha]\right)^{1/\alpha}
&=\mathbb{E}\left[\left(\sum_{i=1}^{Y_{n-1}}(\xi_{n,i}-m)+mY_{n-1}\right)^\alpha\right]^{1/\alpha}\nonumber\\
&\leq ||\sum_{i=1}^{Y_{n-1}}(\xi_{n,i}-m)||_\alpha+m||Y_{n-1}||_{\alpha}:=T_n+my_{n-1}.
\end{align}

Let $Z_{n,i}=\xi_{n}^{(i)}-m$. So $\mathbb{E}[Z_{n,i}]=0,$ $\mathbb{E}[\xi^\alpha]<\infty$ implies $\mu_{\alpha}:=\mathbb{E}[|Z_{n,i}|^\alpha]<\infty$. Then by Lemma \ref{lem:vbe}, 
\begin{align*}
\mathbb{E}\left[\left|\sum_{i=1}^k Z_{n,i}\right|^\alpha\right]\leq 2\sum_{i=1}^k\mathbb{E}[|Z_{n,i}|^\alpha]=2k\mu_{\alpha}.
\end{align*}
Since $Y_{n-1}$ is indepedent of $Z_{n,i}$, we obtain 
\begin{align*}
\mathbb{E}\left[\left|\sum_{i=1}^{Y_{n-1}} Z_{n,i}\right|^\alpha\right]=\E\left[\mathbb{E}\left[\left|\sum_{i=1}^{Y_{n-1}} Z_{n,i}\right|^\alpha \bigg|Y_{n-1}\right]\right]
\leq2\mu_{\alpha}\mathbb{E}[Y_{n-1}]=2\mu_{\alpha}m^{n-1},
\end{align*}
which means $T_n^\alpha\leq 2\mu_\alpha m^{n-1}$, and then $T_n\leq (2\mu_\alpha)^{1/\alpha} m^{\frac{n-1}{\alpha}}$. Substituting into \eqref{yn0}, we obtain 
\begin{align*}
y_n\leq my_{n-1}+ (2\mu_\alpha)^{1/\alpha} m^{\frac{n-1}{\alpha}}.
\end{align*}
Setting $z_n=y_n/m^{\frac{n}{\alpha}}$,  we have
\begin{align*}
z_n\leq m^{1-1/\alpha}z_{n-1}+ (2\mu_\alpha)^{1/\alpha} m^{-1/{\alpha}}.
\end{align*}
Iterating and using $z_0=1$, we obtain
\begin{align*}
z_n\leq m^{n-n/\alpha}z_0+\left(\sum_{j=0}^{n-1}m^{j-j/\alpha}\right)m^{-1/\alpha}(2\mu_\alpha)^{1/\alpha}\leq 1+\frac{m^{-1/\alpha}(2\mu_\alpha)^{1/\alpha}}{1-m^{1-1/\alpha}}:=K<\infty.
\end{align*}
Hence $y_n\leq Km^{n/\alpha}$ holds for $n\geq 0$, and $$\sum\limits_{n=0}\limits^\infty y_n\leq K\sum\limits_{n=0}\limits^\infty(m^{1/\alpha})^n=\dfrac{K}{1-m^{1/\alpha}}<\infty.$$
By Lemma \ref{lem:minkowski}, $\sum_{n=1}^\infty Y_n$ converges almost surly, and 
\begin{align*}
||\tilde{Y}_\infty||_\alpha=||\sum_{n=1}^\infty Y_n||_\alpha\leq \sum_{n=1}^\infty ||Y_n||_\alpha=\sum_{n=1}^\infty y_n<\infty.
\end{align*}
The proof is complete.
\end{proof}

At Chow and Teicher \cite[Page 317]{C and T}, in the case of finite variance,  the authors give a random indexed central limit theorem (Doeblin-Anscombe). In the following we shall prove that, if $\{Z_n,n\geq 1\}$ is a sequence of i.i.d. random variables, with regularly varying tail with index $\alpha$, The following central limit theorems hold.

\begin{theorem} 
\label{tuiguang}
Let $\{Z_n, n\geq 1\}$ be a sequence of i.i.d. random variables whose tails are regularly varying with index $1<\alpha<2$ (as in \eqref{add}), satisfying $\E Z_1=0$. Assume there exist $p, q \in [0, 1]$ such that 
\begin{align} \label{taillr}
\frac{\P(Z>x)}{\P(|Z|>x)} \to p, \quad \frac{\P(Z<-x)}{\P(|Z|>x)} \to q, \quad x\to \infty.
\end{align}
Let $\{t_n, n\geq1\}$ be a sequence of positive, integer-valued random variables. Suppose there exists a sequence of positive integers $\{b_n\}$ and a positive constant $c$ such that $t_n/b_n \overset{\P}{\to} c$ and $b_n \to \infty$. Define $S_n = \sum_{i=1}^n Z_i$. Then, as $n \to \infty$,
\begin{align*}
\frac{S_{t_n}}{a_{t_n}} \overset{d}{\to} S_\alpha(\rho, 0),
\end{align*}
where $\{a_n\}$ is a sequence satisfying $n a_n^{-\alpha}L(a_n) \to 1$, $\rho = p - q$, and $S_\alpha(\rho, 0)$ is an $\alpha$-stable random variable with characteristic function
\[
\exp\left(-|t|^\alpha \Gamma(1-\alpha)\cos\frac{\pi\alpha}{2} \left(1 + i\rho\, \mathrm{sgn}(t)\tan\frac{\pi\alpha}{2}\right)\right).
\]

\end{theorem}

\begin{proof}
Denote\ $Z \stackrel{d}{=}Z_1$. By Lemma \ref{sclt}, the sequence $\{a_n\}$ satisfies $S_{n}/{a_n} \stackrel{d}{\rightarrow} S_\alpha(\rho,0).$ 
As the discussion in  Remark \ref{remark1}, there exists a slowly varying function $\mathcal{L}$ such that
\begin{align}\label{sv00}  a_n\sim  n^{\frac{1}{\alpha}}\mathcal{L}(n),\quad n\to \infty,\end{align}
 Denote $k_n=\lfloor cb_n\rfloor$. We complete the proof in two steps:

{\it Step 1.} Prove $ a_{k_n}/a_{t_n} \stackrel{\P}{\rightarrow}1$. Clearly, for any $\varepsilon>0$, there exist $p,\delta\in (0,1)$ such that
\begin{align}\label{epd}
 \dfrac{(1-p)^2} {1+p}  \left(1-\delta\right)^{\frac{1}{\alpha}+\frac{1}{2} }> 1-\varepsilon,\quad 
\dfrac{(1+p)^2}{1-p} \left(1+\delta\right)^{\frac{1}{\alpha}+\frac{1}{2} }<1+\varepsilon,
\end{align}
Note that $  t_n/k_n \stackrel{\P}{\rightarrow}1$, $k_n\to\infty$. For the above  $\delta\in(0,1)$, as $n\to \infty$,  $$ \P\left(\left|\dfrac{t_n}{k_n}-1\right|> \delta\right)\to 0. $$
By  \eqref{sv00}, for the above  $p\in(0,1)$,  there exists $S>0$, such that for $ n>S$, $$\left|  \frac{a_n}{ n^{1/\alpha}\mathcal{L}(n)}-1\right|<p. $$
For $n>S, m>S$, we have \begin{align}\label{LmLn}\frac{1-p}{1+p}\frac{ n^{1/\alpha}\mathcal{L}(n)}{ m^{1/\alpha}\mathcal{L}(m)}<\frac{a_n}{a_m}<\frac{1+p}{1-p}\frac{ n^{1/\alpha}\mathcal{L}(n)}{ m^{1/\alpha}\mathcal{L}(m)}.\end{align}
By  \cite[Theorem 1.2.1]{RV}, for the above $\delta$, 
\begin{align*}
 \sup_{\lambda\in [1-\delta,1+\delta]}\left|\frac{\mathcal{L}(\lambda x)}{\mathcal{L}(x)}-1\right|\to 0, \quad x\to\infty.
\end{align*}
Then there exists $M>0$,  for $x\geq M$, $\lambda\in [1-\delta,1+\delta]$, 
\begin{align}\label{Llambda}
1-p<\frac{\mathcal{L}(\lambda x)}{\mathcal{L}(x)}<1+p.
\end{align}
Since $k_n\to\infty$, there exists $N$,  for $n>N$,  $ (1-\delta)k_n>\max\{M,S\}$. If $\left|\dfrac{t_n}{k_n}-1\right|< \delta$,  then $ t_n, k_n>\max\{M,S\}$. Substituting \eqref{Llambda} into \eqref{LmLn}, we get 
\begin{align}
\dfrac{(1-p)^2} {1+p}  \left(1-\delta\right)^{\frac{1}{\alpha}} 
<\dfrac{a_{t_n}}{a_{k_n}}< \dfrac{(1+p)^2}{1-p} \left(1+\delta\right)^{\frac{1}{\alpha}} 
\end{align}
holds for $n>N$. Recalling \eqref{epd}, we obtain for $n>N$, $$\left\{\left|\dfrac{a_{t_n}}{a_{k_n}}-1\right|>  \varepsilon, \left|\dfrac{t_n}{k_n}-1\right|< \delta\right\}=\emptyset. $$
 Hence, as $n\to \infty$, 
\begin{align*}
\P\left(\left|\frac{a_{t_n}}{a_{k_n}}-1\right|>\varepsilon\right)
= \P\left(\left|\frac{t_n}{k_n}-1\right|\geq \delta,\left|\frac{a_{t_n}}{a_{k_n}}-1\right|>\varepsilon\right) \le 
\P\left(\left|\frac{t_n}{k_n}-1\right|\geq \delta\right)\to 0.
\end{align*}

{\it Step 2.} Prove $\dfrac{S_{t_n}-S_{k_n}}{a_{k_n}} \stackrel{\P}{\rightarrow}0$.

For any $\varepsilon>0, \delta>0$,
\begin{align}\label{I123}
\P(|S_{t_n}-S_{k_n}|>\varepsilon a_{k_n})& \leq \P(|t_n-k_n|>\delta k_n)+\P\left(\max_{|j-k_n|\leq \delta k_n}|S_j-S_{k_n}|>\varepsilon a_{k_n}\right)\nonumber\\
&\le \P(|t_n-k_n|>\delta k_n)+\P\left(\max_{1\leq m\leq M}|S_{k_n+m}-S_{k_n}|>\varepsilon a_{k_n}\right)
\nonumber\\
& +\P\left(\max_{1\leq m\leq M}|S_{k_n-m}-S_{k_n}|>\varepsilon a_{k_n}\right)\nonumber\\
&=:I_1+I_2+I_3,
\end{align}
where $M:=\lfloor\delta k_n\rfloor$. By $\lim_{n\to\infty} t_n/k_n \stackrel{\P}{\rightarrow}1$ we know that  $\lim_{n\to \infty} I_1=0$. If we have proved that there exists $ T>0$ such that for any  $\delta>0$ it hold that 
\begin{align}\label{I_2}
\limsup_{n\to\infty}I_2\leq T\delta,
\end{align}
analogously we have $\limsup_{n\to\infty}I_3\leq T\delta$.
First letting $n\to \infty$ in \eqref{I123} and then letting $\delta\to 0$,  we get $\dfrac{S_ {t_n}-S_{k_n}}{a_{k_n}} \stackrel{\P}{\rightarrow}0$.

  So, in the following we prove   \eqref{I_2}. Set $u_n=a_{k_n}$. By $\lim\limits_ {n\to\infty}a_n=\infty$ we know that $u_n\to\infty$. Define $Y_{n,i}=Z_i\mathbf{1}_{\{|Z_i|\leq u_n\}}-\E[Z_i\mathbf{1}_{\{|Z_i|\leq u_n\}}]$, $Z_{n,i}=Z_i-Y_{n,i}$. By our assumption, $\mu:=\E|Z|<\infty$. Together with $\E Y_{n,i}=0$, we have $\E Z_ {n,i}=0$. Let
\begin{align*}
S_m^{(Y)}=\sum_{i=k_n+1}^{k_n+m}Y_{n,i},\quad S_m^{(Z)}=\sum_{i=k_n+1}^{k_n+m}Z_{n,i}.
\end{align*}
Then
\begin{align}
\max_{1\leq m\leq M}|S_{k_n+m}-S_{k_n}|\leq\max_{1\leq m\leq M}|S_m^{(Y)}|+\max_{1\leq m\leq M}|S_m^{(Z)}|.
\label{kakak}
\end{align}

Firstly, we estimate the second term on the righthand side of  \eqref{kakak}. Since
$$ Z_{n,i}=  Z_{i}\mathbf{1}_{\{|Z_i|\geq u_n \}}-\E|Z_{i}\mathbf{1}_{\{|Z_i|\geq u_n \}}|, $$
we have 
$$\max_{1\le m\leq M}|S_m^{(Z)}|\leq M\E|Z\mathbf{1}_{\{|Z|\geq u_n \}}| +\sum_{i=k_n+1}^{k_n+M}|Z_{i}\mathbf{1}_{\{|Z_i|\geq u_n \}}|$$

By Markov inequality, 
\begin{align*}
\P\left(\max_{1\leq m\leq M}|S_m^{(Z)}|>\frac{\varepsilon}{2}a_{k_n}\right)&\leq\frac{2}{\varepsilon a_{k_n}}\cdot 2M\E|Z\mathbf{1}_{\{|Z|\geq u_n \}}| 
=\frac{4 M}{\varepsilon a_{k_n}}\int_{u_n}^\infty \P(|Z|>t)\text{d}t,
\end{align*}
Using \eqref{add} and Lemma \ref{Kara}, we have  $$\int_{u_n}^\infty \P(|Z|>t)\text{d}t\sim \frac{u_n^{1-\alpha}}{\alpha-1}L(u_n), $$ Since $u_n=a_{k_n}$, $k_n a_{k_n}^{-\alpha}L(a_{k_n})\to 1$, we then get

\begin{align}\label{Smz}
\P\left(\max_{1\leq m\leq M}|S_m^{(Z)}|>\frac{\varepsilon}{2}a_{k_n}\right)
\sim \frac{4\delta k_n}{\varepsilon a_{k_n}}\frac{1}{\alpha-1}a_{k_n}k_n^{-1} 
= \frac{4}{\alpha-1}\cdot\frac{\delta}{\varepsilon},\quad n\to \infty.
\end{align}

Secondly, we estimate the first term on the righthand side of  \eqref{kakak}. Note that for fixed $n$,  $\{Y_{n,i},i\ge 1\}$ are independent and $\E Y_{n,i}=0$. It is easy to see that $|Y_{n,i}|\leq 2u_n$. By above discussion, $$ |\E Z_i\mathbf{1}_{\{|Z_i|\leq u_n\}}|= |-\E Z_i\mathbf{1}_{\{|Z_i|> u_n\}}|\le   \E [|Z_i| \mathbf{1}_{\{|Z_i|> u_n\}}]  \sim \dfrac{1}{\alpha-1}u_n^{1-\alpha}L(u_n)=O(a_{k_n}k_n^{-1}), $$
which and   Lemma \ref{Kara} yield that as $n\to\infty$,
\begin{align*}
Var(Y_{n,i})&=\int_{0}^{u_n}2t\P(|Z_i|>t)\text{d}t-\left(\E Z_i\mathbf{1}_{\{|Z_i|\leq u_n\}}\right)^2\\
&\sim \frac{2}{2-\alpha}u_n^{2-\alpha}L(u_n)-\left(\frac{1}{\alpha-1}u_n^{1-\alpha}L(u_n)\right)^2\\
&\sim\frac{2 }{2-\alpha}a_{k_n}^2k_n^{-1}.
\end{align*}
By Kolmogorov's inequality, \begin{align*}
\P\left(\max_{1\leq m\leq M}|S_m^{(Y)}|>\frac{\varepsilon}{2}a_{k_n}\right)&\leq\frac{Var\left(\sum\limits_{i=1}\limits^M Y_{n,i}\right)}{(\varepsilon a_{k_n}/2)^2}=\frac{M\cdot  Var(Y_{n,1}) }{(\varepsilon a_{k_n}/2)^2}\\
&\sim \frac{4}{\varepsilon^2 a_{k_n}^2}\cdot\delta k_n \frac{2 }{2-\alpha}a_{k_n}^2k_n^{-1}\\
&\sim\frac{8 }{2-\alpha}\cdot\frac{\delta }{\varepsilon^2 }.
\end{align*}
Combining this with \eqref{Smz}, we find that there exists $T(\varepsilon)$ such that for sufficiently large $n$, it holds that
 \begin{align*}
&  \P\left(\max_{1\leq m\leq M}|S_{k_n+m}-S_{k_n}|>\varepsilon a_{k_n}\right)\\
&\leq \P\left(\max_{1\leq m\leq M}|S_m^{(Y)}|+\max_{1\leq m\leq M}|S_m^{(Z)}|>\varepsilon a_{k_n}\right)\\
&\leq \P\left(\max_{1\leq m\leq M}|S_m^{(Y)}|>\frac{\varepsilon}{2}a_{k_n}\right)+\P\left(\max_{1\leq m\leq M}|S_m^{(Z)}|>\frac{\varepsilon}{2}a_{k_n}\right)\\
&\leq T(\varepsilon)\delta, 
\end{align*}
 i.e.,  \eqref{I_2} holds. In summary,
recalling \eqref{I123}, we complete the proof of  Theorem \ref{tuiguang}.
\end{proof}

{\it Proof of Theorem \ref{th1}.}

By Renewal theory,  as $n\to \infty$, $K_n\to\infty$, $a.s.$ 
Suppose $\{U_i, i\in\mathbb{N}\}$ is the i.i.d. sequence defined by  \eqref{tauU}. By
   Kulkarni and Pakes \cite[pp.475]{BGT}, 
$$U_1\stackrel{d}{=}\sum\limits_{j=1}\limits^I \tilde{Y}_\infty^{(j)},$$
where\ $\{Y^{(i)}_\infty\}$ is an i.i.d. sequence and independent of  $I$, $Y^{(1)}_\infty \stackrel{d}{=} \tilde{Y}_\infty=\sum\limits_{i=0}^\infty Y_i$.  
By Wald's equality and  \eqref{EYinfty} we have $\E U_1=\beta/(1-m)$. 
Denote $\tilde{U}_i=U_i-\E U_i$. 
  By Lemma \ref{D.3}  and \eqref{EYinfty}, as $x\to \infty$, 
\begin{align}\label{U1t0}\P(U_1>x)=\P(\sum\limits_{j=1}\limits^I \tilde{Y}_\infty^{(j)}>x)\sim \P(I>x)(\E(\tilde{Y}_\infty))^\kappa= \P(I>x)(1-m)^{-\kappa},\end{align}
hence $\P(|\tilde{U}_i|>x)\sim \P(I>x)(1-m)^{-\kappa}.$ From \eqref{Kn0}, 
applying Theorem \ref{tuiguang}, we obtain
\begin{align*}
\frac{\sum\limits_{i=1}\limits^{K_n} {U}_i-K_n\beta/(1-m)}{a_{K_n}}=\frac{\sum\limits_{i=1}\limits^{K_n} \tilde{U}_i}{a_{K_n}}\stackrel{d}{\rightarrow} S_\kappa (1,0),\quad n\to \infty. 
\end{align*}
 where $a_n\sim n^\frac{1}{\kappa}\mathcal{L}(n)$.  By Step 1, $\lim\limits_ {n\to\infty}\dfrac{a_{K_n}}{a_n}=\pi_0^{1/\kappa}$. Recalling the decomposition \eqref{div1}, by Lemma 1 of Kulkarni and Pakes \cite{BGT}, the term $\mathbf{1}_{\{B_n\neq \infty\}}\sum_{i=n-B_n}^n X_i$ converges in probability.
 The proof is completed. 
\qed

To prove Theorem \ref{th02}, first we prove
\begin{lemma}
\label{yinli}
If $\mathbb{P}(\xi>x)\sim c x^{-\alpha}l(x)$, where $1<\alpha<2$, $c>0$, $l(x)$ is a slowly varying function at $\infty$, then $\P(\tilde{Y}_\infty >x)\sim\frac{c}{(1-m)^{\alpha+1}}x^{-\alpha}l(x).$
\end{lemma}

\begin{proof}
 This result was established by Guo and Hong \cite[Lemma 2.4]{gh2024}, whose proof relies on a tail estimate for random sums \cite{Asmussen18}. In this paper, we present a direct proof based on analytical methods.

Recall $\varphi(s)=\E[s^{\tilde{Y}_\infty-1}]$.  Recalling $\mu = \E[\tilde{Y}_\infty-1]$, using Lemma \ref{xiy0} and Lemma \ref{BinghamDoney} we can see that $$\varphi(\mathrm{e}^{-\lamb})  = 1 - \mu\lamb + o(\lamb),\quad \lambda\downarrow 0.$$ 
Setting $g(\lamb)=-\log \varphi(\mathrm{e}^{-\lamb})$, then $g(\lamb)\ge0$ and  $g(0)=0$. Since $\varphi'(1)=\mu$,\,we have as $\lamb\to 0$, $g(\lamb) \sim \mu\lamb$. Let
$z(\lamb) := \lamb + g(\lamb).$   Then $z(\lambda) \downarrow   0,\, \lambda  \downarrow  0$. 
By  \eqref{varphi},  
\begin{equation}
\mathrm{e}^{-g(\lamb)} = f\big(\mathrm{e}^{-z(\lamb)}\big).
\end{equation}
 In Lemma \ref{BinghamDoney}, letting $n=1$, we have the expansion of  $f(\mathrm{e}^{-z(\lambda)})$:
\begin{align*}
f(\mathrm{e}^{-z(\lamb)}) &= 1 - m z(\lambda) - c\Gamma(1-\alpha)\,z(\lambda)^\alpha l(1/z(\lambda)) + o\big(z(\lambda)^\alpha l(1/z(\lambda))\big),\quad  \lambda \downarrow 0.
\end{align*} 
By Taylor's expansion, for $1<\alpha<2$, 
$$
\mathrm{e}^{-g(\lamb)}=1-g(\lamb)+O(\lamb^2)=1-g(\lamb)+o\big(\lamb^\alpha l(1/\lamb)\big),\quad  \lambda \downarrow 0.
$$
Hence
\[
1 - g(\lamb) + o\big(\lamb^\alpha l(1/\lamb)\big) = 1 - m z(\lamb) - c\Gamma(1-\alpha)\,z(\lamb)^\alpha l(1/z(\lamb)) + o\big(z(\lamb)^\alpha l(1/z(\lamb))\big).
\]
Note that $z(\lamb)\sim \lamb/(1-m)$ and $l$ is slowly varying, we can replace $l(1/z(\lamb))$ by $l(1/\lamb)$. Some calculation leads to 
\begin{align*}
g(\lamb) &= m z(\lamb) + c\Gamma(1-\alpha)\,z(\lamb)^\alpha l(1/\lamb) + o\big(\lamb^\alpha l(1/\lamb)\big)\\
&= m(\lamb + g(\lamb)) +c\Gamma(1-\alpha)\big(\lamb + g(\lamb)\big)^\alpha l(1/\lamb) + o\big(\lamb^\alpha l(1/\lamb)\big) \notag,\quad  \lambda \downarrow 0.\end{align*}
Then
\begin{align} 
  (1-m)g(\lamb) &= m\lamb +c\Gamma(1-\alpha)\big(\lamb + g(\lamb)\big)^\alpha l(1/\lamb) + o\big(\lamb^\alpha l(1/\lamb)\big), \quad \lambda \downarrow 0. \label{eq:gapprox2}
\end{align}
Recalling $g(\lamb)\sim \mu\lamb = \dfrac{m}{1-m}\lamb$, we obtain
\[
\lamb + g(\lamb)= \frac{1}{1-m}\lamb + o(\lamb),\quad \lambda \downarrow 0.
\]
Therefore,  
\[
\big(\lamb + g(\lambda)\big)^\alpha  = (1-m)^{-\alpha} \lamb^\alpha (1+o(1)),\quad \lambda \downarrow 0.
\]
Substituting this into \eqref{eq:gapprox2}, we obtain 
\[
(1-m)g(\lamb) = m\lamb + c\Gamma(1-\alpha) (1-m)^{-\alpha} \lamb^\alpha l(1/\lamb) + o\big(\lamb^\alpha l(1/\lamb)\big),\quad \lambda \downarrow 0.
\]
We get the expansion formula of  $g(\lamb)$:
\begin{equation}\label{eq:gexp}
g(\lamb) = \frac{m}{1-m}\lamb + \frac{c\Gamma(1-\alpha)}{(1-m)^{\alpha+1}} \lamb^\alpha l(1/\lamb) + o\big(\lamb^\alpha l(1/\lamb)\big),\quad \lambda \downarrow 0.
\end{equation}
Returning to $\varphi(\mathrm{e}^{-\lamb}) = \mathrm{e}^{-g(\lamb)}$, by Taylor's expansion,
\begin{align*}
\varphi(\mathrm{e}^{-\lamb}) &=\sum_{i=0}^\infty (-1)^i\frac{g(\lambda)^i}{i!}  \\
&= 1 - \frac{m}{1-m}\lamb - \frac{c\Gamma(1-\alpha)}{(1-m)^{\alpha+1}} \lamb^\alpha l(1/\lamb) + o\big(\lamb^\alpha l(1/\lamb)\big),\quad \lambda \downarrow 0.
\end{align*}
where we use $g(\lamb)^2 = O(\lamb^2) = o(\lamb^\alpha l(1/\lamb))$. By $\mu = \dfrac{m}{1-m}$, 
we have 
\begin{equation*}
\varphi(\mathrm{e}^{-\lamb}) = 1 - \mu\lamb - \frac{c\Gamma(1-\alpha)}{(1-m)^{\alpha+1}} \,\lamb^\alpha l(1/\lamb) + o\big(\lamb^\alpha l(1/\lamb)\big),  \quad \lambda \downarrow 0.
\end{equation*}
Now applying Lemma \ref{regular} to $\tilde{Y}_\infty$, where $\mu = \E[\tilde{Y}_\infty]-1$, $\alpha\in(1,2)$, we obtain
\[
\P(\tilde{Y}_\infty > x) \sim \frac{c}{(1-m)^{\alpha+1}} \, x^{-\alpha} l(x), \qquad x\to\infty,
\]
completing the proof of Lemma \ref{yinli}.
\end{proof}


{\it Proof of Theorem  \ref{th02}.} 

By Lemma \ref{D.3}, we know that as $x\to \infty$, 
$$\P(U_i>x)=\P\left(\sum_{i=1}^I \tilde{Y}_\infty^{(i)}>x\right)= \E(I)\P(\tilde{Y}_\infty>x)\sim\frac{\beta}{(1-m)^{\alpha+1}} x^{-\alpha} l_2(x).$$
The remaining proof follows the same lines as that of Theorem \ref{th1}, so we omit here. \qed

{\it Proofs of Theorems \ref{th3} and \ref{th4}.} Here we only prove Theorems \ref{th3}, and the proof of Theorems \ref{th4} is similar. For $s\in [0,1]$, we have $1-f(s)\le m(1-s)$, which and iteration lead to  $1-f_k(s)\le m^k(1-s)$. Then 
\begin{align}\label{Ttail}
\P(\tau_1>   k)=1-h(f_{k-1}(0))\le \beta m^{k-1},
\end{align} 
which means that $\E \tau_1^2<\infty$. Applying Feller \cite[XI, 5]{feller1971}, We have the central limit theorem for $K_n$:
$$
\frac{K_n-\pi_0 n}{\sqrt{n(Var \tau_1)\pi_0^3}}\stackrel{d}{\to} N(0,1),\quad n\to \infty.
$$ 
By $na_n^{-\kappa}l_1(a_n)\to (1-m)^{\kappa}$ we have $a_n\sim n^{1/\kappa} \mathcal{L}(n)$ with some slowly varying function $\mathcal{L}$. For $\kappa\in (1,2)$, we have $\sqrt{n}=o(a_n)$.  Therefore,
$$
\frac{K_n-\pi_0 n}{a_{n}}\stackrel{\P}{\to} 0,\quad n\to \infty,
$$ 
Noticing 
\begin{align*}
\frac{S_n-  n\beta\pi_0/(1-m)}{a_n}=\frac{S_n-\beta K_n/(1-m)}{a_n}+\frac{K_n-\pi_0 n}{a_{n}}\cdot \frac{\beta}{1-m},
\end{align*}
and applying Theorem \ref{th1}, we complete the proof. \qed

 From Bingham et al \cite[subsection 8.3.5]{RV}, we have 
\begin{lemma}
\label{lemma-3.5}
Let $\{Z,Z_i:i\ge 1\}$ be a sequence of i.i.d. non-negative random variables and satisfies 
$$\P(Z>x)\sim \frac{l(x)}{x^\alpha\Gamma(1-\alpha)},\quad x\to\infty, $$  
where $0<\alpha<1$, $l(x)$ is a function slowly varying at infinity.  Then as $n\to \infty$,  
\begin{align*}
\frac{S_n}{a_n}:=\frac{\sum\limits_{i=1}\limits^n Z_i}{a_n}\stackrel{d}{\rightarrow}   T,
\end{align*}
where  $\{a_n\}$ is an integer-valued sequence satisfying $na_n^{-\alpha}l(a_n)\to1$, and $T$ is a positive random variable with Laplace transform $\mathrm{e}^{-s^\alpha}$ ($s>0$).
\end{lemma}

Then we have
\begin{theorem}
\label{th3.6}
Suppose that $\{Z,Z_i:i\ge 1\}$ and $S_n$ are defined as in Lemma \ref{lemma-3.5}. 
Let $\{t_n:n\geq1\}$ be a sequence of positive integer-valued random variables, and suppose there exist a sequence of positive integers $\{b_n\}$ and a positive constant $c$ such that $t_n/b_n\stackrel{\P}{\rightarrow}c$ and $b_n\to \infty$. Then as $n\to \infty$,  
\begin{align*}
\frac{S_{t_n}}{a_{t_n}}\stackrel{d}{\rightarrow}   T,
\end{align*}
where  $\{a_n\}$ is an integer-valued sequence satisfying $na_n^{-\alpha}l(a_n)\to1
$, and $T$ is a positive random variable with Laplace transform $\mathrm{e}^{-s^\alpha}$ ($s>0$).
\end{theorem}

\it Proofs of Theorem \ref{th3.6}. 
By Lemma \ref{lemma-3.5}, there exist  $\{a_n\}$ satisfying  $na_n^{-\alpha}l(a_n)\to1$ such that $S_{n}/{a_n} \stackrel{d}{\rightarrow} T.$  Defining $\varkappa(x)=\left(l(x)\right)^{-1/\alpha}$ and $\mathcal{L}_1(x)=[\varkappa^\#(x^ {1/\alpha})]^{-1}$, then $a_n\sim n^{1/\alpha}\mathcal{L}_1(n)$. 
 Denote $k_n=\lfloor cb_n\rfloor$. As in the proof of Theorem \ref{tuiguang}, it is not difficult to get    $ a_{k_n}/a_{t_n} \stackrel{\P}{\rightarrow}1$. 

Next,  we shall prove $\dfrac{S_{t_n}-S_{k_n}}{a_{k_n}} \stackrel{\P}{\rightarrow}0$.

For any $\varepsilon>0, 0<\delta<1$, noting that $\{Z_i:i\ge 1\}$ are non-negative, 
\begin{align}\label{alpha123}
\P(|S_{t_n}-S_{k_n}|>\varepsilon a_{k_n})
&\le \P(|t_n-k_n|>\delta k_n)+\P\left(S_{\lfloor(1+\delta)k_n\rfloor}-S_{k_n}>\varepsilon a_{k_n}\right)
\nonumber\\
& +\P\left(S_{k_n}-S_{\lfloor(1-\delta)k_n\rfloor}>\varepsilon a_{k_n}\right)\nonumber\\
&=:I_1+I_2+I_3.
\end{align}
Clearly, $\lim_{n\to \infty} I_1=0$. 

By $a_n\sim n^{1/\alpha}\mathcal{L}_1(n)$, there exists $N$ such that for $n>N$, $$\frac{a_{k_n}}{a_{ \lfloor\delta k_n\rfloor}}\ge \frac{1}{2}\delta^{-\frac{1}{\alpha}}\ge \frac{1}{2\delta}. $$
Then for $I_2$, we have
\begin{align}
I_2&= \P\left(\sum_{i=k_n+1}^{\lfloor(1+\delta)k_n\rfloor} Z_i>\varepsilon a_{k_n}\right)
=\P\left(\sum_{i=1}^{ \lfloor\delta k_n\rfloor} Z_i>\varepsilon a_{k_n}\right) \nonumber\\
&=\P\left(\frac{\sum_{i=1}^{ \lfloor\delta k_n\rfloor} Z_i}{a_{ \lfloor\delta k_n\rfloor}}>\varepsilon\cdot \frac{ a_{k_n}}{a_{ \lfloor\delta k_n\rfloor}}\right)\nonumber\\
&\le \P\left(\frac{\sum_{i=1}^{ \lfloor\delta k_n\rfloor} Z_i}{a_{ \lfloor\delta k_n\rfloor}}>\frac{\varepsilon}{2\delta}  \right)\nonumber\\
&\to \P(T>\frac{\varepsilon}{2\delta} ),\quad n\to \infty,
\end{align}
where the last step is by  Lemma \ref{lemma-3.5}. The estimate for $I_3$ is similar. Letting $n\to \infty $ in \eqref{alpha123}, we obtain
$$
\limsup_{n\to \infty}\P(|S_{t_n}-S_{k_n}|>\varepsilon a_{k_n})\le 2\P(T>\frac{\varepsilon}{2\delta} ).
$$
Then letting $\delta\to 0$,  we complete the proof of  Theorem \ref{th3.6}.
 
\qed

{\it Proof of Theorem \ref{1.6}. }  
By Lemma \ref{D.3} and \eqref{EYinfty}, as $x\to \infty$, 
\begin{align}\label{U1t}\P(U_1>x)=\P\left(\sum\limits_{j=1}\limits^I \tilde{Y}_\infty^{(j)}>x\right)\sim \P(I>x)(\E(\tilde{Y}_\infty))^\gamma=\frac{1}{(1-m)^\gamma}\P(I>x)\sim\frac{x^{-\gamma}l_3(x)}{(1-m)^\gamma},\end{align}
Using \eqref{Kn0} and 
applying Theorem \ref{th3.6}, we obtain
\begin{align*}
\frac{\sum\limits_{i=1}\limits^{K_n} {U}_i}{a_{K_n}}\stackrel{d}{\rightarrow} T,
\end{align*}
where $a_n$ satisfies $na_n^{-\gamma}l_3(a_n)\to(1-m)^\gamma[\Gamma(1-\gamma)]^{-1}$. 
  By Step 1 in the proof of Theorem \ref{tuiguang}, $\lim\limits_ {n\to\infty}\dfrac{a_{K_n}}{a_n}=\pi_0^{1/\gamma}$. Recalling the decomposition
\eqref{div1},  similarly as in the proof of Theorem \ref{th1}, we obtain the result.
\qed


%

\appendix\section{Appendix}

\begin{lemma}[Minkowski inequality]\label{lem:minkowski} \cite[Theorem 9.2a]{bai}

Let $\{Z_n\}_{n\ge 0}$ be a sequence of random variables and $\alpha \ge 1$. If $\E[|Z_n|^\alpha] < \infty$ for every $n$ and the series $\sum\limits_{n=0}\limits^\infty \|Z_n\|_\alpha$ converges, then $\sum\limits_{n=0}\limits^\infty Z_n$ converges absolutely almost surely, and
\[
\Bigl\|\sum_{n=0}^\infty Z_n\Bigr\|_\alpha \le \sum_{n=0}^\infty \|Z_n\|_\alpha,
\]
where $\|Z\|_\alpha := (\E[|Z|^\alpha])^{1/\alpha}$.
\end{lemma}

\begin{lemma}[von Bahr--Esseen inequality]\label{lem:vbe} \cite[Theorem 9.3a]{bai}
Let $Z_1, \dots, Z_k$ be independent random variables satisfying $\E[Z_i] = 0$ and $\E[|Z_i|^\alpha] < \infty$, where $\alpha \in [1,2]$. Then
\[
\E\Bigl[\Bigl|\sum_{i=1}^k Z_i\Bigr|^\alpha\Bigr] \le 2 \sum_{i=1}^k \E[|Z_i|^\alpha].
\]
\end{lemma}

\begin{lemma}[Karamata's Integration Theorem]\cite{RV}\label{Kara}
\label{Karamata}
Let $L(x)$ be a function slowly varying at infinity. For $\rho>-1$, we have:
\begin{align*}
\int_0^x t^\rho L(t)\text{d} t\sim \frac{x^{\rho+1}L(x)}{\rho+1},\quad x\to\infty.
\end{align*}
For $\rho<-1$, we have:
\begin{align*}
\int_x^\infty t^\rho L(t)\text{d} t\sim \frac{x^{\rho+1}L(x)}{-(\rho+1)},\quad x\to\infty.
\end{align*}
\end{lemma}

 \begin{lemma}[Stable Central Limit Theorem] (\cite[XVII.5, Theorem 2]{feller1971}\label{sclt})
If $\{Z_n\}_{n\geq 1}$ is a sequence of independent random variables with zero mean and regularly varying tails of the form \eqref{add} with index $\alpha$ ($1<\alpha<2$), and if there exist $ p,q\in [0,1]$ such that
\begin{align} \label{taillr}
\frac{\P(Z>x)}{\P(|Z|>x)}\to  p,\quad \frac{\P(Z<-x)}{\P(|Z|>x)}\to  q, \quad\quad x\rightarrow \infty,
\end{align}
then
\begin{align*}
\frac{S_n}{a_n} \stackrel{d}{\rightarrow} S_\alpha(\rho,0),
\end{align*}
where $\{a_n\}$ is a sequence satisfying $n a_n^{-\alpha}L(a_n)\to 1$, $\rho=p-q$, and $S_\alpha(\rho,0)$ is a random variable following an $\alpha$-stable distribution with the following characteristic function:
$$\exp\left(-|t|^\alpha \Gamma(1-\alpha)\cos\frac{\pi\alpha}{2} (1+i\rho\mathrm{sgn}(t)\tan\frac{\pi\alpha}{2})\right).  $$
\end{lemma}

\begin{lemma}[{\cite[Theorem 8.1.6]{RV}}]\label{BinghamDoney} 
$\xi$ is an non-negative random variable, $l(\cdot)$ is a slowly varying function at $\infty$. Define $F(x)=\P(\xi\le x)$. $\hat F$ is the Laplace transform of $\xi$, i.e., $\hat F(s)=\E[\mathrm{e}^{-s \xi}]$.   If $\mu_n:=\E[\xi^n]<\infty,\, \alpha=n+\beta,\, \beta\in[0,1]$, then the following are equivalent:
\begin{align*}
f_n(s):= (-1)^{n+1}\left\{\hat F(s)-\sum_{i=0}^n \mu_i\frac{(-s)^i}{i!} \right\}\sim s^\alpha l(1/s),\,  s\downarrow 0\\
\begin{cases}
\displaystyle\int_{x}^\infty t^n dF(t)\sim n!l(x)\quad   &\beta=0 \\
1-F(x)\sim \dfrac{(-1)^n}{\Gamma (1-\alpha)}x^{-\alpha}l(x)   &\beta\in(0,1),\quad  x\to \infty,\\ 
\displaystyle\int_{0}^x t^{n+1} dF(t)\sim (n+1)!l(x)\quad   &\beta=1
\end{cases}
\end{align*}
\label{regular}
\end{lemma}

 \begin{lemma} [{\cite[Proposition D.3]{Barczy}}]\label{D.3}

Let $\tau$ be a non-negative integer-valued random variable and let  $\{\zeta, \zeta_i,i\ge 1\}$ be a sequence of  i.i.d. non-negative random variables, independent of $\tau$, such that $\tau$ has regularly varying tail of index $\kappa\in [0,\infty)$ and $0<\E(\zeta) <\infty$. In case of $\kappa\in [1,\infty)$, assume additionally that  there exists $r>\kappa$ satisfying $\mathbb{E}(\zeta^r)<\infty$. Then we have
$$
\P\left(\sum_{i=1}^\tau \zeta_i>x\right)\sim (\E \zeta)^\kappa\P(\tau>x),\quad x\to \infty.$$
\end{lemma}

 \begin{lemma}[{\cite[Proposition 4.1]{Fay06}}]\label{D.4} 
 Let $\tau$ be a non-negative integer-valued random variable and let  $\{\zeta, \zeta_i,i\ge 1\}$ be a sequence of  i.i.d. non-negative random variables, independent of $\tau$,  such that $\zeta$ has regularly varying tail of index $\alpha\in (0,\infty)$ and $\P(\tau>x)=o(\P(\zeta>x))$ as $x\to \infty$. Then
$$\P\left(\sum_{i=1}^\tau \zeta_i>x\right)\sim \E \tau \P(\zeta>x),\quad x\to \infty.$$ 
\end{lemma}

 {\bf Acknowledgement.}  The second author is supported by National Natural Science Foundation of China  (Grant No. 12271043).

\end{document}